\documentclass[11pt]{amsart}
\usepackage{fullpage}
\usepackage{color}
\usepackage{graphicx,psfrag} %,epstopdf}
\usepackage{color} %This loads the color package
\usepackage{tikz}
\usepackage{pgffor}
\usepackage{hyperref}
\usepackage{todonotes}
\usepackage{eqnarray,amsmath}

\usepackage{perpage}     %the perpage package

\usepackage{wrapfig}
\usepackage{enumerate}
\usepackage{subfigure}
\usepackage{caption}

\usepackage[normalem]{ulem}
\usepackage[makeroom]{cancel} %%%%cross-out in math mode
\usepackage{amssymb} %\nmid

\MakePerPage{footnote}   %the perpage package command

\newtheorem{theorem}{Theorem}[section]
\newtheorem{lemma}[theorem]{Lemma}
\newtheorem{corollary}[theorem]{Corollary}

\theoremstyle{definition}
\newtheorem{definition}[theorem]{Definition}
\newtheorem{example}[theorem]{Example}
\theoremstyle{remark}

\newtheorem{ex-prop}[theorem]{Example-Proposition}

\numberwithin{equation}{section}

\def\real{\mathbb{R}}
\def\P{\mathbb{P}}
\def\E{\mathbb{E}}
\def\L{\mathcal{L}}
\def\M{\mathcal{M}}
\def\gN{\gamma_{\real^N}}

\def\MjN{\M_j^{\gN}}

\newcommand{\beq}{\begin{eqnarray}}
\newcommand{\eeq}{\end{eqnarray}}
\newcommand{\beqq}{\begin{eqnarray*}}
\newcommand{\eeqq}{\end{eqnarray*}}

\def\convas{\stackrel{a.s. }{\to}}

\def\:{:\,}

\newcommand{\definedas}{\stackrel{\Delta}{=}}
\def\C{\mathbb C}

\begin{document}

\title{The Intrinsic Geometry of Some Random Manifolds}
%\title{Random knots using Chebyshev billiard table diagrams}

\author{Sunder Ram Krishnan}
\address{Andrew and Erna Viterbi Faculty of Electrical Engineering, Technion -- Israel Institute of Technology, Haifa 32000, Israel}
%\texttt{\MakeLowercase{\indent }}}
\email{eeksunderram@gmail.com\newline\indent
https://sites.google.com/site/eeksunderram/}
%\thanks{The first author was supported in part by the funding from the European Research Council under the European Union's Seventh Framework Programme, Grant FP7-ICT-318493-STREP}
\thanks{Research supported in part by URSAT, ERC Advanced Grant 320422.}

\author{Jonathan E. Taylor}
\address{Department of Statistics, Stanford University, Stanford, CA, 94305-4065}
\email{jonathan.taylor@stanford.edu\newline\indent http://www-stat.stanford.edu/\%7Ejtaylor/}
\thanks{Research supported in part by SATA,   AFOSR, FA9550-11-1-0216.}

\author{Robert J. Adler}
\address{Andrew and Erna Viterbi Faculty of Electrical Engineering, Technion -- Israel Institute of Technology, Haifa 32000, Israel}
\email{robert@ee.technion.ac.il\newline\indent http://webee.technion.ac.il/people/adler/}
\thanks{Research supported in part by SATA II, AFOSR FA9550-15-1-0032 and
   URSAT, ERC Advanced Grant 320422.}

\begin{abstract}
We study the a.s.\ convergence of a sequence of random embeddings of a fixed manifold into Euclidean spaces of increasing dimensions.  We show that the limit is deterministic. As a consequence, we show that many intrinsic functionals of the
embedded manifolds also converge to deterministic limits. Particularly interesting examples of these functionals are  given by the  Lipschitz-Killing curvatures, for which we also prove  unbiasedness,  using the Gaussian kinematic formula.
\end{abstract}


\maketitle

\section{Introduction}
\label{sec:intro}
In the recent paper \cite{Krishnan} we studied the limiting behaviour of the global reach of a  sequence of random manifolds embedded in Euclidean spheres of increasing dimensions. To be precise, we proved that the global reaches of these  random manifolds  converge, almost surely (a.s.), to a  deterministic constant that had arisen earlier in other scenarios, specifically in the theory of Gaussian extremes.  In this paper we look more closely at these random embeddings, and show that the results of \cite{Krishnan} can be extended to show the convergence   not only of the reaches of the embedded manifolds, but, in an appropriate sense, of the manifolds themselves, along with their induced Riemannian structures.

More specifically, we consider the following setup,  effectively equivalent to that in \cite{Krishnan}. We start with a centered, unit variance, smooth Gaussian process $f$ on a compact, smooth manifold $M$ (the precise assumptions on $f$ and $M$ are stated in the following section). We let $f_1,f_2,\dots$ be a sequence of independent copies of $f$, 
set $f^k=(f_1,\dots,f_k)$, and define an embedding $h^k$ of $M$ into $\mathbb{R}^k$ by
\begin{equation}
h^k(x)\ = \ \frac{1}{\sqrt{k}}f^k(x) \ =\ \frac{1}{\sqrt{k}}(f_1(x),\cdots,f_k(x)),
\label{eq:embed}
\end{equation}
for all $x\in M$. 
By the Whitney embedding theorem and regularity assumptions on $f$ and $M$ to follow, we are assured of an a.s.\
 embedding as long as  $k$ is large enough. ($k>2\text{dim}(M)$ will suffice.)
 
Our initial aim was to analyse the limiting behaviour of certain  functionals  defined on the random, embedded 
manifolds $h^k(M)$. In particular, if we equip each  $h^k(M)$ with the Riemannian metric,  $g^k_E$ say, that it inherits as a subset of
$\mathbb{R}^k$, then we were particularly interested in intrinsic functionals; viz.\ those that depend only on the metric. The basic question was whether or not such functionals would  converge to the corresponding intrinsic functional evaluated on $(M,g)$, for an appropriately chosen metric $g$ on $M$.

Choosing $g$ correctly, this turns out to be true, and the underlying reason  is 
 the fact that the Riemannian manifolds $(h^k(M),g^k_E)$ themselves converge, in an appropriate sense, to $(M,g)$. 
 
 In the following section we make the notions of ``correctly" and ``in an appropriate sense" precise, by describing some basic results on Gaussian processes and the convergence of manifolds. There we also state the 
 main result of the paper, Theorem \ref{thm:main}, about the convergence of the random 
 Riemannian manifolds  $(h^k(M),g_E^k)$.  The a.s.\ convergence of a family of intrinsic functionals to deterministic constants follows  as a corollary. 
 
 In Section \ref{sec:GKF}
 we focus on  a particular family of functionals, such as  volume and  surface area,   that come under the title of `Lipschitz-Killing curvatures' (LKCs), and describe their convergence to their `intuitive' limits. We also note that the a.s.\ limit in this case is also the 
 ($k$-independent) expected value of the corresponding LKC of each of the random manifolds $h^k(M)$. In other words, we show the unbiasedness of the LKCs. 
 
Section \ref{sec:proof}
contains the proof of Theorem \ref{thm:main} and its corollary, and the final Section \ref{sec:LKC} contains
 the proofs of the results in Section  \ref{sec:GKF}.  
%
% of the above will be made mathematically precise in the sections to follow. First, we present minimal details on the topics of Gaussian processes on manifolds and manifold convergence in Section \ref{sec:prelim}, and follow it up with statements of our main results. The geometry necessary to follow the paper is presented in a separate section (cf. Section \ref{sec:GKF}). Those familiar with curvature integrals and the GKF may skip this portion. In the following Section \ref{sec:proof}, we prove the claim of the main theorem and an important corollary. The result of this corollary is illustrated in Section \ref{sec:LKC} by considering  a.s. convergence of the sequence of LKCs, where we also show that these functionals are unbiased.

We shall not say much about motivation in this paper. In \cite{Krishnan} we discussed, in the context of reach, our reasons for studying random manifolds,  many of which came from questions arising in theorems about learning the homology of  manifolds from point cloud data sampled from them. While the discussion there centered on the reach of the  $h^k(M)$ (or, more precisely, a version of the $h^k(M)$ embedded in spheres) it applies equally well to the issues treated in this paper. Thus we refer the interested reader to \cite{Krishnan} for details.

\section{Some preliminaries}
\label{sec:prelim}
Before we can state our main result, we need to set up some notation and quote some basic results relating to  Gaussian processes on manifolds and to the convergence of Riemannian manifolds.

To start, we shall assume that the $m$-dimensional  Riemannian manifold $(M,g)$ is  $C^3$, connected, oriented, boundaryless,  and compact, so that it has a finite atlas. That is, $M$ can be covered by a finite number  of open sets $\Omega_i$, and there exist smooth, one to one maps $\varphi_i\:\Omega_i\to U_i\subset \real^m$, 
% or patches, denoted by $\Omega_1,\cdots,\Omega_N$.
for $i=1,\dots,N$. When working in charts on $M$,  $\left(\frac{\partial}{\partial x^1},\cdots,\frac{\partial}{\partial x^m}\right)$ denotes a coordinate basis for the tangent space $T_x M$. We use the standard notation $g_{ij}=g\left(\frac{\partial}{\partial x^i},\frac{\partial}{\partial x^j}\right),\,\,1\leq i,j\leq m$. $\nabla$ denotes the Levi-Civita connection of $(M,g)$, and $\nabla^2$  the corresponding covariant Hessian.  Note that, 
when convenient, we shall adopt Einstein summation conventions.

\subsection{Gaussian processes on Riemannian manifolds}
A zero mean, real valued Gaussian process, $f:M\rightarrow\mathbb{R}$, is determined by its covariance function $\C\:M\times M\to\mathbb{R}$ given by
$$
\C(x,y)\ = \ E\{f(x)f(y)\},
$$ 
which is assumed to be positive definite on $M\times M$ and smooth enough so that the sample paths of $f$ are a.s.\ $C^3$ on $M$. We also assume that the joint distributions of $f$ and its derivatives are non-degenerate. From Corollary 11.3.5 in \cite{Adler}, this implies that the sample paths of $f$ are a.s.\ Morse over $M$. 

Such processes induce a Riemannian metric, $g^\C$, on the tangent bundle $T(M)$ of $M$, defined by  
\begin{equation}
g^\C_x(X,Y)\ \definedas\ \E \{(Xf)(x)\, (Yf)(x)\}\  =\ Y_y X_x \C(x,y)\big|_{y=x},
\label{eq:indmet}
\end{equation}
where $X,Y$ are vector fields with values $X_x,Y_x\in T_x M$. The assumptions above on $\C$, particularly its  positive definiteness,  guarantee that $g^\C$ is a non-degenerate, well defined metric. We call $g^\C$ as the {\it metric induced by $f$} \cite{Adler}.

Throughout this paper we shall assume that $g\equiv g^\C$, which we do either by starting with the Riemannian manifold $(M,g)$ and then choosing the Gaussian process appropriately, or by starting with $M$ and $\C$, and then choosing $g$ as $g^\C$. 
Thus, from now on, we shall use only the metric $g$, and assume that it is also the one induced by $\C$.
% 
%In the scenario where no metric is specified on $M$, we assume that it is endowed with the process induced metric. On the other hand, if we are given a pair $(M,g)$, the Gaussian field is so chosen that the metric induced by (\ref{eq:indmet}) is exactly $g$ (cf.\ proof of Theorem 12.6.1 in \cite{Adler}). 
This notation, and the smoothness assumptions above on $f$ and $M$,  are assumed to hold throughout the paper. 

%As was mentioned in Section \ref{sec:intro}, our main result (cf. Theorem \ref{thm:main} below) talks about the convergence of  manifolds. This convergence is what we explain next.

\subsection{Convergence of Riemannian manifolds}	
To define the convergence of a sequence of Riemannian manifolds $(M_k,g_k)$ to a limit $(M,g)$, we follow Section 10.3 of  \cite{Petersen}, applied to our situation, in which all manifolds are compact.  (Consequently we do not require the notion of `pointed' manifolds, which appears  in \cite{Petersen}.) 

We start with a norm from which follows a notion of function space convergence for real valued functions, $u\:M\to\real$.
%The topology for  convergence is defined is a function space topology involving derivatives. 
%We take the usual uniform norms of the function and its derivatives up to the $i$th order.
With  $\{(\Omega_\ell,\varphi_\ell)\}_{\ell=1}^N$ an atlas for $M$, adopt  multi-index notation 
 ${j}=(j_1,\cdots,j_m)$,  $|{j}|=j_1+\cdots+j_m$, to write,  for $u\:\Omega_\ell\to\real$, 
\beqq
\partial^j u\ =\ \partial^{j_1}_{1} \cdots\partial^{j_m}_{m} u= \frac{\partial^{|j|}u}{\partial(x^{1})^{j_{1}} \cdots\partial(x^{m})^{j_{m}}}.
\eeqq
We then define the $C^i$ norm of $u$ on $M$ as
\begin{equation}
\|u\|_i \ =\ \ \max_{1\leq\ell\leq N} \left( \sup_{x\in\Omega_\ell}|u(x)|+ \sum_{1\leq |j|\leq i}\ \sup_{x\in\Omega_\ell} \left|\partial^j u(x)\right|  \right).
\label{eq:norm}
\end{equation}
When there is no possibility of confusion, we shall typically not write the index; i.e.\ we shall write $\|u\|$ rather than $\|u\|_i$.
%The corresponding  $C^i$ norm for functions  on manifolds, $f\:M\to\real$, with the atlas described above, is then defined as
%\beq
%\label{norm}
%\|u\|_{i,M} \  = \max_\ell \left\|   \varphi_\ell^*u      \right\|_{i,\Omega_\ell},
%\eeq
%where $\varphi_\ell^*u$ is the pullback of $u$ of the $\ell$-th patch to $\Omega_\ell$. 

We can now formulate two definitions.

\begin{definition}
\label{def:gconverge}
A sequence of Riemannian metrics  $g_k$ on a $C^i$ manifold $M$ is said to converge in the $C^i$ topology to a metric $g$ if the real valued functions  $(g_k)_{ij}$ converge to the $g_{ij}$  on $M$, in the  $C^i$ topology.
\label{def:converge}
\end{definition}

\begin{definition}
\label{def:Mconverge}
A sequence of compact, $C^i$,  Riemannian manifolds $(M_k,g_k)$ is said to converge in the $C^i$ topology to a $C^i$ manifold $(M,g)$ if, for large enough $k$,  we can find $C^i$ embeddings $H_k \: M\rightarrow M_k$ such that the pullbacks $H_k^*g_k$ converge to  $g$ on $M$ in the  $C^i$ topology.\label{def:converge}
\end{definition}

As shown in \cite{Petersen}, neither of the above notions of  convergence is dependent on the choice of atlas. Furthermore, treating the manifolds as metric spaces with the metric being Riemannian distance, the second definition implies 
Gromov-Hausdorff convergence.

In our scenario, one can take the embeddings $H_k=h^k$ so that $M_k=h^k(M)$ with the Euclidean metric in $\mathbb{R}^k$ it inherits, implying $g_k=g^k_E$.  We now have all the background we need for stating our first theorem.

\section{Convergence of Gaussian manifolds}
\label{cgm}
\subsection{The main results}

\begin{theorem}
Let $(M,g)$ be a connected, orientable, compact, $C^3$ Riemannian manifold, and $f\:M\to\real$ a zero mean Gaussian process
with a.s.\ $C^3$ sample paths inducing the metric $g$. Let  $h^k:M\rightarrow\mathbb{R}^k$ be the embedding of $M$ defined
by \eqref{eq:embed}, and $g_E^k$ denote the metric induced on $h^k(M)$ by the Euclidean metric in $\mathbb{R}^k$. Then,
with probability one,
\beq
\left(h^k(M),\,g^k_E\right)\ \stackrel{C^2}{\longrightarrow}\ (M,g),
\label{main:convergence}
\eeq
where the convergence  is as in  Definition \ref{def:converge}.
\label{thm:main}
\end{theorem}

The a.s.\ convergence of intrinsic functionals, described in the Introduction, will  now follow as a simple corollary of Theorem \ref{thm:main}, once we have the right definitions. To this end, let $M$ be a compact, $C^i$ manifold, and $\mathcal G^i$ the collection of all $C^i$ metrics on $M$, with the topology induced by the convergence in Definition \ref{def:gconverge}. We say that   $F_M\:\mathcal G^i\to\real$ is a $C^i$ intrinsic functional on $M$ if it is continuous with respect to this topology. 

\begin{corollary}
Retain the notation and assumptions of Theorem \ref{thm:main}, and let $F_{M}$ be a $C^2$ intrinsic functional $F$ of  $M$. 
Then 
\beq
F_{h^k(M)}\left(g_E^k\right)\ \stackrel{a.s.}{\longrightarrow}\ F_M(g).
\label{cor:intrin}
\eeq
\end{corollary}

Before turning to the proofs of these results, note that the main result of \cite{Krishnan}, which established the a.s.\ convergence
of the reaches of the embedded manifolds $h^k(M)$, follows from neither of these. One reason for this is that the embeddings used there were slightly different to those used in this paper, in that they were self-normalized and so mapped into spheres. The main reason, however, is that the reach has both global and local aspects, and so is not an intrinsic functional of a manifold, either in the sense of the above definition or any other reasonable replacement for it.

Another point worth noting is that the proof will show that had we only assumed $C^1$  in each place where we assumed $C^3$, this would suffice to establish \eqref{main:convergence} with sup norm convergence, which has the consequence that the mapping $(M,g)\to (h^k(M), g^k_E)$ is an  asymptotically isometric embedding. This is a result of independent interest, and already mentioned in \cite{Krishnan}. 

On the other hand, if we were to  assume $C^n$  in each place where we assumed $C^3$, this would suffice to establish \eqref{main:convergence} with $C^{n-1}$ convergence. No significant change to the proof is required. Our statement of Theorem \ref{thm:main}, in between these two extremes, was motivated by the examples we had in mind, most of which involve curvatures, and so $C^2$ functionals, but nothing beyond that.

%\begin{remark}
%\begin{enumerate}
%\item The embedding map $h^k$ defined in (\ref{eq:embed}) differs from the one we used in \cite{Krishnan} in that there we normalised by the norm of the vector $(f_1(x),\cdots,f_k(x))$, so that the resulting point sat in $S^{k-1}$. Since it is trivial that $\|(f_1(x),\cdots,f_k(x))\|/\sqrt{k}$ converges a.s. to $1$, there is actually nothing to be gained by studying the map as in \cite{Krishnan}, not to mention the unnecessarily complicated calculus involved. On the other hand, we were convinced that there was nothing to be gained in terms of computations by using the map in (\ref{eq:embed}) when we studied the a.s. behaviour of the global reach. That is, the delicate zero over zero issues there would still persist.\\
%\item The result about a.s. convergence of the global reach of $h^k(M)$ to a well known deterministic constant does not follow from Corollary \ref{cor:intrin}. This is because reach is a functional that not only incorporates local curvature information, but also the global topology such as the proximity of two points on the manifold that are geodesically far apart. Therefore, reach is not intrinsic; that is, it is not describable only in terms of the metric tensor. The value of reach depends heavily on the embedding.
%\end{enumerate}
%\end{remark}

\subsection{Proof of Theorem \ref{thm:main}}
\label{sec:proof}

Our proof will rely heavily on  standard   limit theory for Banach space valued random variables. In particular, we shall 
exploit Corollary 7.10 of \cite{Talagrand-Ledoux}, which we now 
quote for the reader's convenience. 

\begin{theorem}[\cite{Talagrand-Ledoux}, Corollary 7.10]
\label{TL:theorem}
Let $X$ be a Borel random variable with values in a separable Banach space $B$, with norm $\|\cdot\|_{\text{B}}$. Let $S_n$ be the partial sum of $n$ i.i.d. realizations of $X.$ Then, $$\frac{S_n}{n}\ \stackrel{a.s.}{\longrightarrow} \ 0,
$$
 if, and only if, $E \{\|X\|_{\text{B}}\}<\infty$ and $E \{X\}=0$.
%\label{t1}
\end{theorem}

To apply this in our setting, recall Definition \ref{def:converge} of convergence of a sequence of compact manifolds and the fact that we work in coordinate patches denoted by $\Omega_1,\cdots,\Omega_N$ on $M$. We are interested in proving that $((h^k)^*g^k_E)_{ij}\convas g_{ij}$ in the $C^2$ topology (cf.\ \eqref{eq:norm}). At $x\in M$, the components of the pullback tensor in the coordinate frame $\left(\frac{\partial}{\partial x^1},\cdots,\frac{\partial}{\partial x^m}\right)$ are
\begin{eqnarray}
((h^k)^*g^k_E)_{ij}(x)&=&(h^k)^*g^k_E\left(\frac{\partial}{\partial x^i},\frac{\partial}{\partial x^j}\right)\nonumber\\
&=&g^k_E\left(h^k_\ast\frac{\partial}{\partial x^i},h^k_\ast\frac{\partial}{\partial x^j}\right)\nonumber\\
&=&\frac{1}{k}\sum_{\ell=1}^k\frac{\partial f_\ell(x)}{\partial x^i}\frac{\partial f_\ell(x)}{\partial x^j}.
\label{eq:pullback}
\end{eqnarray}
An immediate consequence of this and the fact that $g$ is the induced metric for $f$ (cf.\ \eqref{eq:indmet} and the discussion following it) is that 
\beq
\label{mean:equn}
\E\left\{(h^k)^*g^k_E)_{ij}(x)\right\} \ = \ g_{ij}(x),
\eeq
for all $x\in M$.

To apply Theorem \ref{TL:theorem} in our setting, take 
\beq
\label{D1}
X\ =\ \frac{\partial f}{\partial x^i}\frac{\partial f}{\partial x^j} - g_{ij}, 
\eeq
and set the Banach space $B$ to be $C^2(M)$ (twice continuously differentiable functions over $M$) along with the norm given by (\ref{eq:norm})). 

Then the mean zero condition of  Theorem  \ref{TL:theorem}  is trivial, and we need only show the finiteness of 
$\E\{\|X\|\}$. This norm depends on the derivatives of $X$ up to second order,  and so, at the risk of being accused of  
being overly pendantic, we write out what the random parts of these derivatives actually are. (The non-random parts involve derivatives of  $\C$, and since it  and its derivatives are assumed to be uniformly continuous over  $M$
there is nothing to check here.)

Performing covariant differentiation with respect to the vector field $\frac{\partial}{\partial x^p}$, it is easily seen that the first order derivative equals
\beq
\label{D2}
\nabla^2 f(x)\left(\frac{\partial}{\partial x^p},\frac{\partial}{\partial x^i}\right)\frac{\partial f(x)}{\partial x^j}+\frac{\partial f(x)}{\partial x^i}\nabla^2 f(x)\left(\frac{\partial}{\partial x^p},\frac{\partial}{\partial x^j}\right),
\eeq
where we use the 2-form notation for the covariant Hessian, and remind the reader that $\nabla$ is the Levi-Civita connection associated with $g$.

Recalling the definition of the covariant Hessian,  $\nabla^2f(X,Y)=g(\nabla_X\nabla f,Y)$, we obtain the following expression for the typical second order derivative:

%\begin{equation*}
%\bigg[\left(\frac{\partial}{\partial x^q}\frac{\partial}{\partial x^p}\frac{\partial}{\partial x^i}-\frac{\partial}{\partial x^q}\nabla_{\frac{\partial}{\partial x^p}}\frac{\partial}{\partial x^i}\right)f(x)+2\nabla^2 f(x)\left(\frac{\partial}{\partial x^p},\frac{\partial}{\partial x^i}\right)\nabla^2 f(x)\left(\frac{\partial}{\partial x^q},\frac{\partial}{\partial x^j}\right)
%\end{equation*}
%\begin{equation*}
%\quad+\left(\frac{\partial}{\partial x^q}\frac{\partial}{\partial x^p}\frac{\partial}{\partial x^j}-\frac{\partial}{\partial x^q}\nabla_{\frac{\partial}{\partial x^p}}\frac{\partial}{\partial x^j}\right)f(x)\bigg].
%\end{equation*}

\beq
\label{D3}
&&\left(\frac{\partial}{\partial x^q}\frac{\partial}{\partial x^p}\frac{\partial}{\partial x^i}-\frac{\partial}{\partial x^q}\nabla_{\frac{\partial}{\partial x^p}}\frac{\partial}{\partial x^i}\right)f(x)\frac{\partial f(x)}{\partial x^j}
\\ &&\qquad\qquad+\nabla^2 f(x)\left(\frac{\partial}{\partial x^p},\frac{\partial}{\partial x^i}\right)\nabla^2 f(x)\left(\frac{\partial}{\partial x^q},\frac{\partial}{\partial x^j}\right)
\nonumber
\\ \nonumber
&&\qquad\qquad\qquad+\nabla^2 f(x)\left(\frac{\partial}{\partial x^q},\frac{\partial}{\partial x^i}\right)\nabla^2 f(x)\left(\frac{\partial}{\partial x^p},\frac{\partial}{\partial x^j}\right)
\\
&&\qquad\qquad\qquad \qquad+\frac{\partial f(x)}{\partial x^i}\left(\frac{\partial}{\partial x^q}\frac{\partial}{\partial x^p}\frac{\partial}{\partial x^j}-\frac{\partial}{\partial x^q}\nabla_{\frac{\partial}{\partial x^p}}\frac{\partial}{\partial x^j}\right)f(x) . \nonumber
\eeq

The norm, $\|X\|_B$, that we need now involves taking the supremum norm of each expression in \eqref{D1}--\eqref{D3} over a chart, summing over $p$ and $q$, and then taking the maximum over all charts. However, despite the complicated expressions here, all that appears are derivatives, of up to third order, of the Gaussian process $f$, which we have assumed to have a.s.\ continuous (Gaussian!) derivatives of up to order three. It thus immediately follows from (occasionally multiple) applications of 
the Cauchy-Schwarz inequality, along with the Borel-Tsirelson-Ibragimov-Sudakov inequality (e.g\  \cite{Adler}, Theorem 2.1.2), that $\E\|X\|_B<\infty$, with room to spare.  (In fact, the BTIS inequality gives the finiteness of exponential moments of $X$.)

This finiteness, along with Theorem  \ref{TL:theorem}, completes the proof of Theorem \ref{thm:main}.

%
%
%and the sentences before it.). Similar arguments lead us to the end of the proof of the following a.s. convergence in $C^2(M)$ for arbitrary, but fixed $i,j$ in each chart:
%$$((h^k)^*g^k_E)_{ij}(x)\stackrel{a.s.}{\longrightarrow}g_{ij}(x),\,\,\,\text{uniformly in $x\in M$}.$$
%
\subsection{Proof of Corollary \ref{cor:intrin}}
From Theorem \ref{thm:main}, it is now trivial that a functional $F$ continuously dependent only on the Riemannian metric and its first and second derivatives converges a.s. in each chart. If the functional involves integrating over the whole of $M$, we simply resort to the standard partition of unity argument to lift local results to the global scenario in conjunction with one of the convergence theorems from the theory of Lebesgue integration.  (This is illustrated by the example of the LKCs in the next section.).

%Presently, we follow the path laid out by Corollary \ref{cor:intrin} for the case of $\mathcal{L}_j(h^k(M))$.

\section{Lipschitz-Killing curvatures and the Gaussian kinematic formula}
\label{sec:GKF}
In this section we will give a cursory introduction to LKCs and the Gaussian kinematic formula (GKF), with the aim of making the results of the following section meaningful.  A full theory of both LKCs and the GKF can be found in \cite{Adler}, or the more user friendly Saint-Flour notes \cite{stflour}.

\subsection{Lipschitz-Killing curvatures}
Nice Euclidean sets $A$ of dimension $N$ have $N+1$ LKCs, $\L_0(A),\dots,\L_N(A)$. Of these, $\L_N(A)$ is the $N$-dimensional volume of $A$,   $\L_{N-1}(A)$ is proportional to its  $(N-1)$-dimensional surface area, and $\L_0(A)$ is its Euler characteristic.  The remaining LKCs are somewhat harder to describe, although, in a somewhat ill defined sense, they are often considered to be  measures of `the $k$-dimensional size' of $A$. Perhaps the easiest way to introduce them is via a tube formula of the form
\beq
\lambda_N(\text{Tube}(A,\rho))\ =\ \sum_{j=0}^N\rho^{N-j}\omega_{N-j}\mathcal{L}_j(A).
\label{tube}
\eeq
Here $\lambda_N$ is Lebesgue measure in $\real^N$, the `tube' Tube($A,\rho$) around $A$ is the set of all points in $\real^N$ of distance not more than $\rho$ from 
$A$,  and $\omega_{N-j}$ is the volume of the unit ball in $\mathbb{R}^{N-j}$.  The tube formula \eqref{tube} holds for all $\rho$ less than the reach of $A$, where the reach is precisely the object we studied in \cite{Krishnan}. 
The expansion \eqref{tube} holds for a large class of nice sets (such as locally convex, Whitney stratified submanifolds in $\real^N$), and so provides a definition of the LKCs. However, when $A$ is a smooth, $m$-dimensional manifold, $M$, satisfying the conditions of this paper,  there is also a rather simple, direct, integral representation of the LKCs, given by 
\begin{equation}
\mathcal{L}_j(M)= \begin{cases} \frac{(-2\pi)^{-(m-j)/2}}{\left(\frac{m-j}{2}\right)!}\int_M\text{Tr}(R^{(m-j)/2})\text{Vol}_{g} &\mbox{if $m-j$ is even}  \\ 
0 & \mbox{if $m-j$ is odd}. \end{cases} 
\label{eq:LKC}
\end{equation}
Here Vol$_{g}$ denotes the volume form on $(M,g)$, where $g$ is the Riemannian metric induced on $M$ by its embedding in Euclidean space,  and $R$ is the 
 Riemannian curvature tensor.  Since $R$ can be considered as a double form of type $(2,2)$, it  makes  sense 
  to talk about its powers, and  their trace, Tr. (Details can be found in  Chapters 7--10 of  \cite{Adler}.)

One of the first points to note from the representation \eqref{eq:LKC} is that since the integral depends only on the volume form, determined by the metric $g_E$, and the curvature tensor $R$, LKCs are intrinsic functionals of 
$M$, dependent  on $g_E$  through its first two derivatives. The second point is that 
there is nothing particularly Euclidean about the integral in \eqref{eq:LKC} and so we could use this as a definition of $\L_j(M)$ for an Riemannian manifold $(M,g)$. In this case, however, the LKCs need not be related to a tube formula such as \eqref{tube}.
For more on LKCs in this more general setting, see either \cite{Adler} or the more recent and extensive results on valuations in, for example, \cite{AlFu}.

%
%
%\begin{remark}
%\item It is a beautiful theorem in Differential Topology named the Chern-Gauss-Bonnet theorem, which proves that the intrinsic definition in (\ref{eq:LKC}) for the case of $j=0$ equals the Euler characteristic. We point out the strength of this result: You can deform the manifold $M$ in a smooth manner so that the topology does not change. The Euler characteristic, being a topological invariant, cannot change. However, the integrand in (\ref{eq:LKC}) can vary a lot! Even then, the integral is bound to equal the Euler characteristic.
%\label{rem:CGB}
%\end{remark}

\subsection{Gaussian Minkowski functionals}
In the setting of Integral Geometry it is customary to work not directly with LKCs, but rather with a renumbered and scaled version of them known as (Lebesgue) Minkowski functionals, defined by 
\begin{equation}
\mathcal{M}_j(A)\ \definedas \ j!\omega_j\mathcal{L}_{N-j}(A), \qquad j=0,\dots,N.
\label{eq:mink}
\end{equation}
In terms of these functionals, the tube formula \eqref{tube} becomes
\beq
\lambda_N(\text{Tube}(A,\rho))\ =\ \sum_{j=0}^N\frac{\rho^{j}}{j!}\mathcal{M}_j(A),
\label{Mtube}
\eeq
which is, basically, a standard (but finite!) Taylor series expansion of the tube volume as a function of $\rho$. As before, $A$ must be `nice' and $\rho$ must be small enough.

A superficially similar expansion holds if we replace the Lebesgue measure
$\lambda_N$ by the standard Gaussian measure on $\real^N$, which we denote by
 $\gN$. In this case we have the following (cf.\  \cite{Adler}  Theorem 10.9.5 and Corollary 10.9.6).
\begin{equation}
\gamma_{\mathbb{R}^N}(\text{Tube}(A,\rho))=\gamma_{\mathbb{R}^N}(A)+\sum_{j=1}^\infty\frac{\rho^j}{j!}\mathcal{M}_j^{\gamma_{\mathbb{R}^N}}(A),
\label{tubes:equ}
\end{equation}
where the   $\MjN (A)$  are defined by this expansion, for small enough $\rho$, 
 and are known as the 
{\it Gaussian} Minkowski functionals. Note that, as opposed to the regular tube formula, the expansion in the Gaussian case does not terminate after a finite number of terms. Furthermore, the Gaussian Minkowski functionals, unlike their Lebesgue counterparts, are not translation invariant.

In addition to the role they play in the GKF, which will become clear in the following subsection, the main fact that we will need about these functionals is given in the following lemma.

\begin{lemma}
For any linear subspace $S$ of codimension $n\geq 1$ in $\mathbb{R}^k$, the Gaussian Minkowski functionals satisfy,  for all $ j\geq 0$,
\beq
\label{halfspace:equn}
\mathcal{M}_j^{\gamma_{\mathbb{R}^k}}(S)\ =\ \mathcal{M}_j^{\gamma_{\mathbb{R}^n}}(\{0\}).
\eeq
\label{lma:minkid}
Furthermore, for all $j < n$,
\beq
\label{jlessthann:equn}
\mathcal{M}_j^{\gamma_{\mathbb{R}^n}}(\{0\}) \ = \ 0.
\eeq\end{lemma}
\begin{proof}
To prove \eqref{halfspace:equn} assume, without loss of generality,  that 
\beqq
S\ =\ \left\{x\in\real^k\: x_j = 0,\, j=1,\dots, n,\ x_j\in\real,\, j=n+1,\dots,k\right\},
\eeqq
so that 
\beqq
\text{Tube}(S,\rho)  \ =\ \left\{x\in\real^k\: \left\|(x_1,\dots,x_{n})\right\| \leq \rho,\ x_j\in\real,\, j=n+1,\dots,k\right\}
\eeqq
and 
\beq
\label{tubes:equ2}
\text{Tube}(S,\rho)   \ =\ \text{Tube}(\{0\},\rho)\times \real^{k-n},
\eeq
where the origin $0$ here is in $\real^n$.

Computing the Gaussian measure of both sides of \eqref{tubes:equ2} via \eqref{tubes:equ} and comparing coefficients of $\rho$ establishes  \eqref{halfspace:equn}.

As for \eqref{jlessthann:equn}, note that 
\beqq
 \gamma_{\mathbb{R}^n}(\text{Tube}(\{0\},\rho)) \ = \  \P\left\{\chi^2_n \leq \rho^2\right\},
 \eeqq
 where $\chi^2_n$ is a chi-squared random variable with $n$ degrees of freedom. The right hand side here, however, is precisely
 \beq
 \label{expand:equn}
\frac{1}{2^{n/2}\Gamma(n/2)}  \int_0^{\rho^2} x^{n/2-1}e^{-x/2}\,dx \ = \
\frac{1}{2^{n/2}\Gamma(n/2)} \sum_{\ell=0}^\infty  \frac{(-1/2)^\ell}{\ell!} \int_0^{\rho^2} x^{n/2+\ell-1}\,dx, 
\eeq
which gives a power series in $\rho$, the lowest order term of which is $O(\rho^n)$. Comparing coefficients  with the 
expansion \eqref{tubes:equ} establishes \eqref{jlessthann:equn}, as required.

% that since $S$ is a linear subspace of $\mathbb{R}^k$ of dimension $k-n$, the Gaussian measure with respect to $\gamma_{\mathbb{R}^k}$ of Tube$(S,\rho)$ is equal to the Gaussian measure of Tube$(0,\rho)$ with respect to $\gamma_{\mathbb{R}^n}$; that is, there is an averaging over $k-n$ dimensions. The result of the lemma follows from an expansion of $\gamma_{\mathbb{R}^k}(\text{Tube}(S,\rho))$ and $\gamma_{\mathbb{R}^n}(\text{Tube}(0,\rho))$ as given in (\ref{eq:Gautube}) remembering that the Gaussian density is absolutely continuous with respect to Lebesgue measure.
\end{proof}

\subsection{Gaussian kinematic formula}

We now turn to the GKF. Consider the scenario of the Introduction, specifically the (un-normalised) embedding 
$f^k\definedas (f_1,\cdots,f_k)$  of  $M$ into $\real^k$ (cf.\ \eqref{eq:embed}). Although we have assumed that $M$ was a manifold, for the remainder of this subsection  we could  actually take it to be a stratified manifold satisfying the smoothness conditions of Chapter 15  of  \cite{Adler}. Consider the preimage under $f^k$ in $M$ of a regular, stratified manifold 
$D$ in $\mathbb{R}^k$, again satisfying some smoothness conditions that are trivially satisfied if $D$ is assumed to be a compact, $C^2$, manifold.  In the context of deriving mean LKCs of the excursion sets of non-Gaussian fields on manifolds, the following formula,  nowadays referred to as the GKF, was proven in \cite{Adler}.
\begin{equation}
\E \left\{\mathcal{L}_i(M\cap (f^k)^{-1}(D))\right\}\ =\ \sum_{j=0}^{m-i}{i+j \brack j} (2\pi)^{-j/2}\mathcal{L}_{i+j}(M)\mathcal{M}_j^{\gamma_{\mathbb{R}^k}}(D).
\label{eq:GKF}
\end{equation}
where  ${a \brack b}=\binom{a}{b}\frac{\omega_a}{\omega_{a-b}\omega_b}$  are the so-called flag coefficients, and  the LKCs are computed with respect to the metric induced by $f$.

The GKF has myriad applications, but in the following section we shall add an extra, somewhat novel,  one. We shall use it to establish that 
 for fixed, but large enough $k$, and for all $j$, 
  $$\E\left\{\mathcal{L}_j(h^k(M))\right\}\ =\ \mathcal{L}_j(M).
  $$

\section{Convergence of the $\mathcal{L}_j(h^k(M))$, and their unbiasedness}
\label{sec:LKC}
We start 
%by applying Corollary \ref{cor:intrin} to show 
with the a.s.\ convergence of the random variables $\mathcal{L}_j(h^k(M))$. 

\begin{example}
Under the same setup and conditions on $M$ and $f$ as in Theorem \ref{thm:main},
\beq
\label{eq:LKC conv}
\mathcal{L}_j(h^k(M))\stackrel{a.s.}{\longrightarrow}\mathcal{L}_j(M),
\eeq
for each $0\leq j\leq m$.
\label{ex:LKC conv}
\end{example}
\begin{proof}
In view of Corollary  \ref{cor:intrin}, we only need to show that LKCs are $C^2$ intrinsic functionals. For a reader with a background in Differential Geometry, this is (under the conditions we assume) obvious, and so the proof is done.

For the reader without this background, we will provide an outline of a slightly longer proof, which will also introduce issues  relevant to later
discussions.

We start with the 
representation \eqref{eq:LKC} of LKCs, which in our case becomes, for the non-zero case in which $m-j$ is even,  
\begin{equation}
\mathcal{L}_j\left(h^k(M)\right)\ = \ K_j\int_{h^k(M)}\text{Tr}\left((R_E^k)^{(m-j)/2}\right)\text{Vol}_{g_E^k} 
\label{eq:LKC2}
\end{equation}
where $K_j={(-2\pi)^{-(m-j)/2}}/(\left(\frac{m-j}{2}\right)!)$ and $R^k_E$
denotes the curvature tensor of $(h^k(M),g^k_E)$.

However, since for large enough $k$, the embedding map $h^k$ is a diffeomorphism, it follows from the very definition of (global) isometries that $(M,(h^k)^*g^k_E)$ and $(h^k(M),g^k_E)$ are isometric Riemannian manifolds, and so 
\beq
\label{LKC:iso}
\mathcal{L}_j\left(h^k(M)\right)\ = \ K_j\int_{M}\text{Tr}\left((\widetilde R_E^k)^{(m-j)/2}\right)\text{Vol}_{\widetilde g_E^k} ,
\eeq
where we write $\widetilde g_E^k$ to denote the pullback $(h^k)^{*}g_E^k$ and $\widetilde R_E^k$ for the corresponding
curvature tensor, both on $M$.

However, the  Riemannian curvature tensor $R$ on a generic Riemannian manifold  $(M,g)$ is given by (cf.\ \cite{Petersen})
$$ 
R_{ijk \ell}\ =\ \frac{1}{2}\left(\frac{\partial^2 g_{i\ell}}{\partial x^j\partial x^k}+\frac{\partial^2 g_{jk}}{\partial x^i\partial x^\ell}-\frac{\partial^2 g_{ik}}{\partial x^j\partial x^\ell}-\frac{\partial^2 g_{j\ell}}{\partial x^i\partial x^k}\right)+g_{np}(\Gamma^n_{jk}\Gamma^p_{i\ell}-\Gamma^n_{j\ell}\Gamma^p_{ik}),
$$
where the Christoffel symbols of the second kind  are given by
$$
\Gamma^n_{jk}\ =\ \frac{1}{2}g^{n\ell}\left(\frac{\partial g_{\ell j}}{\partial x^k}+\frac{\partial g_{\ell k}}{\partial x^j}-\frac{\partial g_{jk}}{\partial x^\ell}\right),
$$
and the $g^{n\ell}$ are the elements of $G^{-1}$, where $G$ is the matrix with elements $g_{ij}$.

Returning to our current setup, since the symmetric form $g$ is nondegenerate (following from positive definiteness of $\C$)  $G$ is non-singular, and so it is clear is that the components of $R$ depend solely upon the metric tensor and its first and second order derivatives in a smooth manner.   
Since  Theorem \ref{thm:main}  implies the convergence of
$(h^k(M),g^k_E)$ to $(M,g)$, it follows that  
\beq
\label{R:conv}
(R^k_E)_{ijk\ell}(x)\stackrel{a.s.}{\longrightarrow}R_{ijk\ell}(x),\qquad \text{uniformly in $x$}.
\eeq
This, together with  \eqref {eq:LKC2},  \eqref{LKC:iso} and \eqref{R:conv} imply \eqref{eq:LKC conv}, and we are done.

%
%The LKCs are obtained by integrating powers of the trace of the curvature tensor (cf. (\ref{eq:LKC})). Since 
%(e.g.\ (7.2.4) and (7.2.6) in \cite{Adler} )  these powers and traces depend smoothly on $R$ itself, which as we already know, is a smooth function of the metric tensor and its first two derivatives in local coordinates.

%The idea then would be to use a partition of unity $\{u_i\}_{i=1}^N$ subordinate to the open cover $\{\Omega_i\}_{i=1}^N$ of $M$ defined before. Thereafter, taking $\text{Tr}((R^k_E)^p)(x)$ and $\text{Tr}((R)^p)(x)$ to mean the coordinate representations of the $p$th power of the trace: 
%$$\int_{M}\text{Tr}((R^k_E)^p)(x)dx=\int_M\left(\sum_i u_i(x)\right)\text{Tr}((R^k_E)^p)(x)dx=\sum_i\int_{\Omega_i}\left( u_i(x)\text{Tr}((R^k_E)^p)(x)\right)dx.$$
%Since we have established the uniform convergence in charts of $R^k_E$ and since the integrand only involves smooth functions of the metric tensor and its derivatives, a bounded convergence argument yields the desired convergence of LKCs, remembering that the number of charts is finite because of compactness of $M$.
\end{proof}

One of the consequences of \eqref{eq:LKC conv}, the oft-noted fact that $h^k$ is asymptotically isometric,  and some moment checking (with which we shall not bother, for reasons to soon become clear), is that, for each $0\leq j\leq m$, 
$$\lim_{k\to\infty} \E \left\{\mathcal{L}_j(h^k(M))\right\} \ = \ \mathcal{L}_j(M).$$

In fact, we can do better than this. Given the integral representation \eqref{eq:LKC2} of the LKCs of $h^k(M)$, and the subsequent explanations of what all the terms are, we could, in principle at least, take expectations and compute the mean 
$ \E \left\{\mathcal{L}_j(h^k(M))\right\} $ explicitly, for each $k$. One would not expect this calculation to be an easy one.

However, it turns out that there is no need to go this route, since the following result shows that these expectations are actually independent
of $k$, at least for $k$ large enough to ensure a true embedding. We call this  the `unbiasedness' of the LKCs. 

\begin{theorem} Under the conditions of Theorem \ref{thm:main},
for all $k$ for which $h^k$ is an embedding, and for each $0\leq j\leq m$, 
\beqq
\E \left\{\mathcal{L}_j(h^k(M))\right\}\ =\ \mathcal{L}_j(M).
\eeqq
\end{theorem}

\begin{proof}
We start with some generalities.
Let $A$ be  a compact submanifold of dimension $a$, isometrically embedded in some Riemannian manifold $(\tilde{M},\tilde{g})$. Let  $\theta$ be a Gaussian random field on $\tilde{M}$ with induced metric $\tilde{g}$ satisfying the conditions of the GKF and 
define processes $\Theta^n$, for $1\leq n\leq a$,  as $\Theta^n=(\theta_1^n,\cdots,\theta_n^n)$, with the individual components being i.i.d.\ copies of  $\theta$. Then (\ref{eq:GKF}) gives us that
\begin{eqnarray}
\E \{\mathcal{L}_0(A\cap (\Theta^n)^{-1}\{0\})\}&=&\sum_{j=0}^{a}(2\pi)^{-j/2}\mathcal{L}_j(A)\mathcal{M}_j^{\gamma_{\mathbb{R}^n}}(\{0\})\nonumber\\
&=&\sum_{j=n}^{a}(2\pi)^{-j/2}\mathcal{L}_j(A)\mathcal{M}_j^{\gamma_{\mathbb{R}^n}}(\{0\}),
\label{eq:ident}
\end{eqnarray}
where the change in summation limits comes from \eqref{jlessthann:equn}.

Write $\mu_{\chi_\Theta}(A)$ for the $a+1$ vector 
\beqq
\left(\mathcal{L}_0(A),\  \E \left\{\mathcal{L}_0(A\cap (\theta^1)^{-1}\{0\})\right\},\dots,
\E \left\{\mathcal{L}_0(A\cap (\theta^n)^{-1}\{0\right\})\}
\right).
\eeqq
If we adopt the convention that $\Theta^0$ is a function that maps identically to zero,  so that  
\beqq
 \E \{\mathcal{L}_0(A\cap (\theta^0)^{-1}\{0\})\}\ = \  \mathcal{L}_0(A), 
 \eeqq
 then we can rewrite \eqref{eq:ident}, formally, as 
\beqq
\mu_{\chi_\Theta}(\cdot)\ =\  \mathcal{ZL}(\cdot),
\eeqq
where $\mathcal{L}$ maps $A$ to $(\mathcal{L}_0(A),\cdots,\mathcal{L}_{a}(A))$ and 
$\mathcal{Z}$ is a universal $(a+1)\times(a+1)$ upper triangular matrix, the precise elements of which can be found from the 
expansion \eqref{expand:equn}. It is easy to check that the diagonal elements are non-zero, but their precise values are not important for what follows. However, this does imply that $\mathcal{Z}$ is invertible, from which it follows that 
\beqq
\mu_{\chi_\Theta}\ =\ \mathcal{ZL}\quad \iff\quad  \mathcal{L}\ =\ \mathcal{Z}^{-1}\mu_{\chi_\Theta},
\eeqq
so that we can recover the LKCs $(\mathcal{L}_j(A))_{0\leq j\leq a}$ from the expected Euler characteristics  $(\E \{\mathcal{L}_0(A\cap (\Theta^n)^{-1}\{0\})\})_{0\leq n\leq a}$.
%Note that in the case of $n=0$, we already know that $\$ 

%If we now write $\Theta$, or $\Theta_a$, to denote the $a\times a$ matrix with $n$-th row $\theta^n$ followed by $a-n$ zeroes,
%then we can, at least formally, rewrite (\ref{eq:ident}) as
%$$
%\E \{\mathcal{L}_0(\cdot\cap (\Theta)^{-1}\{0\})\}=\mathcal{Z}\mathcal{L}(\cdot),
%$$
%where we understand the left hand side to be a vector made up of the expectations  on the left hand side of \eqref{eq:ident}.
%As for the right hand side, $\mathcal{L}$ maps $A$ to $(\mathcal{L}_0(A),\cdots,\mathcal{L}_{a}(A))$ and 
%$\mathcal{Z}$ is a universal $(a+1)\times(a+1)$ upper triangular matrix, the precise elements of which can be found from the 
%expansion \eqref{expand:equn}, but which are not important for what follows.
%
%%
%diagonal entries $(2\pi)^{-n/2}n!\omega_n$.
% (using (\ref{eq:mink}), $\mathcal{M}_n^{\gamma_{\mathbb{R}^n}}(\{0\})=n!\omega_n\mathcal{L}_0(\{0\})$) and $\mathcal{L}$ maps $A$ to $(\mathcal{L}_0(A),\cdots,\mathcal{L}_{a}(A))$. 

We now exploit the above  to prove the theorem. Firstly, fix $k$, large enough so that $h^k(M)$ is an embedding of $M$ in $\real^k$. 
Set $A=h^k(M)$, with dimension $a=m$. Then  $\real^k$ together with the standard Euclidean metric is the $(\tilde{M},\tilde{g})$ above. 

A simple  way to define  centered, unit variance $\mathbb{R}^n$ valued Gaussian fields $\Theta^n_k$ on $\mathbb{R}^k$ that induce the Euclidean metric is to take a $n\times k$ matrix, $W^n_k$, of i.i.d.\ standard Gaussians, and to set
\beqq
\Theta_k^n(x) \ = \ W_k^nx, \qquad x\in\real^k,
\eeqq
where all our vectors (such as $\Theta_k^n$ and $x$) are written as column vectors. 

The above general argument thus implies that we can compute $(\mathcal{L}_j(h^k(M)))_{0\leq j\leq m}$ from the expected Euler characteristics of the zero sets of $(\Theta^n_k)_{0\leq n\leq m}$ restricted to $h^k(M)$. To do this, note first 
%\footnote{There is a small notational issue here. Since $h^k\:M\to\real^k$
%is a diffeomorphism, $\mathcal{L}_0(h^k(M))=\mathcal{L}_0(M)$, and so for $n=0$ we take $\E \{\mathcal{L}_0(h^k(M)\cap (\theta^0_k)^{-1}\{0\})\}\equiv\L_0(M)$.}.
 the simple, but crucial, fact that,  for $1\leq n\leq m$,  $(\Theta^n_k)^{-1}(\{0\})=\text{null}(W^n_k)$, so that
\begin{eqnarray*}
h^k(M)\cap (\Theta^n_k)^{-1}(\{0\})&=&h^k(M)\cap \text{null}(W^n_k)\\
&=&h^k\left(M\cap(h^k)^{-1}\text{null}(W^n_k)\right)\\
&=&h^k\left(M\cap(f^k)^{-1}\text{null}(W^n_k)\right).
\end{eqnarray*}
Therefore,
\beqq
\sum_{j=0}^{m}
(2\pi)^{-j/2}\mathcal{L}_j(h^k(M))\mathcal{M}_j^{\gamma_{\mathbb{R}^n}}(\{0\})&=&\E _{\Theta^n_k}\left\{\mathcal{L}_0\left(h^k(M)\cap (\theta^n_k)^{-1}(\{0\})\right)\right\}\\
&=&\E _{W^n_k}\left\{\mathcal{L}_0\left(h^k\left(M\cap(f^k)^{-1}\text{null}(W^n_k)\right)\right)\right\}\\
&=&\E _{W^n_k}\left\{\mathcal{L}_0\left(M\cap(f^k)^{-1}\text{null}(W^n_k)\right)\right\},
\eeqq
where the first equality is a direct consequence of the GKF, the second is  from the calculations above, and the last follows from the facts that $h^k$ is a diffeomorphism and the Euler characteristic is a topological invariant.

Consequently, we have that
\begin{eqnarray*}
\sum_{j=0}^{m}(2\pi)^{-j/2}\E _{f^k}\left\{\mathcal{L}_j(h^k(M))\right\}\mathcal{M}_j^{\gamma_{\mathbb{R}^n}}(\{0\})&=&\E _{f^k}\E _{W^n_k}\left\{\mathcal{L}_0\left(M\cap(f^k)^{-1}\text{null}(W^n_k)\right)\right\}\\
&=&\E _{W^n_k}\E _{f^k}\left\{\mathcal{L}_0\left(M\cap(f^k)^{-1}\text{null}(W^n_k)\right)\right\}\\
&=&\sum_{j=0}^m (2\pi)^{-j/2}\mathcal{L}_j(M)\mathcal{M}_j^{\gamma_{\mathbb{R}^k}}(\text{null}(W^n_k)),
\end{eqnarray*}
where the first equality follows from Fubini,  and the last from the GKF. 

However, null$(W^n_k)$ is a linear subspace of codimension $n$ in $\mathbb{R}^k$ for almost every $W^n_k$, and for any linear subspace $S$ of codimension $n$ in $\mathbb{R}^k$, we have from Lemma \ref{lma:minkid} that
$$\mathcal{M}_j^{\gamma_{\mathbb{R}^k}}(S)=\mathcal{M}_j^{\gamma_{\mathbb{R}^n}}(\{0\}).$$
This results in the  identity
$$\sum_{j=0}^m (2\pi)^{-j/2}\mathcal{L}_j(M)\mathcal{M}_j^{\gamma_{\mathbb{R}^k}}(\text{null}(W^n_k))=\sum_{j=0}^m (2\pi)^{-j/2}\mathcal{L}_j(M)\mathcal{M}_j^{\gamma_{\mathbb{R}^n}}(\{0\}),$$
from which follows the fact that
$$\sum_{j=0}^{m}(2\pi)^{-j/2}\E _{f^k}\left\{\mathcal{L}_j(h^k(M))\right\}\mathcal{M}_j^{\gamma_{\mathbb{R}^n}}(\{0\})=\sum_{j=0}^m (2\pi)^{-j/2}\mathcal{L}_j(M)\mathcal{M}_j^{\gamma_{\mathbb{R}^n}}(\{0\}).$$
In matrix formulation, the above reads as
$$\E _{f^k}\{\mathcal{ZL}(h^k(M))\}=\mathcal{ZL}(M),$$
and the  theorem follows on recalling that $\mathcal{Z}$ is invertible.
\end{proof}
%\begin{remark}
%The proof above is concise thanks to the GKF. We point out that one might try to write out the expression for the LKCs of a submanifold of Euclidean space using the formula (10.5.5) in Chapter 10 of \cite{Adler}, and then try to find out its expectation. Even though filling in the small details in the formula such as the second fundamental form and so on is not difficult in our model, taking the expectation of the expression and proving that it is indeed the LKC of $(M,g)$ would be quite a handful.
%\end{remark}

\section*{Acknowledgements}
We would like to thank Sreekar Vadlamani, Shmuel Weinberger, and Yuliy Baryshnikov for useful discussions.

%%%%%%%%%%%%%%%%%%%%%%%%%%%%%%%%%%%%%%%%%%%%%%%%%%%%%%%%%%%%%%%%%%%
%%                                                               %%
%% You may add acknowledgments (optional).                       %%
%%                                                               %%
%%%%%%%%%%%%%%%%%%%%%%%%%%%%%%%%%%%%%%%%%%%%%%%%%%%%%%%%%%%%%%%%%%%

%\bibliographystyle{amsplain}
%\bibliography{yourbibfilename}

% add below the content of your .bbl file produced by bibtex.

%%%%%%%%%%%%%%%%%%%%%%%%%%%%%%%%%%%%%%%%%%%%%%%%%%%%%%%%%%%%%%%%%%%
%%                                                               %%
%% You have reached the end of your document.                    %%
%%                                                               %%
%%%%%%%%%%%%%%%%%%%%%%%%%%%%%%%%%%%%%%%%%%%%%%%%%%%%%%%%%%%%%%%%%%%

\end{document}